\documentclass{amsart}
\usepackage[utf8]{inputenc}
\usepackage[margin=1in]{geometry}
\usepackage{upgreek,amsthm,bbm,bm}
\usepackage{color,amssymb, comment, mathrsfs,tikz,amsmath,amsfonts,stmaryrd,tikz-cd,hyperref,appendix,cancel,faktor,verbatim}
\usepackage{amsmath,amsthm,amssymb}
\usepackage{tikz,tikz-cd,color}
\usepackage{enumitem}
\usepackage{mathtools}
\usepackage{comment}
\usetikzlibrary{matrix}
\usetikzlibrary{arrows}
\usepackage[all]{xy}
\usepackage{amsmath, nccmath}
\usepackage{geometry}

\hypersetup{
    colorlinks=true,
    linkcolor=blue,
    filecolor=magenta,      
    urlcolor=cyan,
} 

\usepackage[capitalise]{cleveref}
\crefformat{equation}{(#2#1#3)}
\crefrangeformat{equation}{(#3#1#4--#5#2#6)}
\crefformat{enumi}{(#2#1#3)}
\crefrangeformat{enumi}{(#3#1#4--#5#2#6)}

\newtheorem{theorem}{Theorem}[section]
\newtheorem{lemma}[theorem]{Lemma}
\newtheorem{proposition}[theorem]{Proposition}
\newtheorem{corollary}[theorem]{Corollary}
\theoremstyle{definition}
\newtheorem{definition}[theorem]{Definition}
\newtheorem{construction}[theorem]{Construction}
\newtheorem{example}[theorem]{Example}

\theoremstyle{remark}
\newtheorem{remark}[theorem]{Remark}

\newtheorem{chunk}[theorem]{}

\newcommand{\kk}{\Bbbk}
\newcommand{\Der}{\mathrm{Der}}
\newcommand{\N}{\mathbb{N}}
\newcommand{\Q}{\mathbb{Q}}
\newcommand{\Z}{\mathbb{Z}}

\newcommand{\T}{\mathbb{T}}
\newcommand{\Hom}{\mathrm{Hom}}
\newcommand{\hooklongrightarrow}{\lhook\joinrel\longrightarrow}
\newcommand{\ii}{\{i\}}
\newcommand{\lp}{\left(}
\newcommand{\rp}{\right)}
\newcommand{\Diff}{\mathrm{Diff}}

\newcommand{\m}{\mathfrak{m}}
\newcommand{\vp}{\varphi}

\title{The Taylor resolution over a skew polynomial ring}
\date{}

\usepackage[capitalise]{cleveref}
\crefformat{equation}{(#2#1#3)}
\crefrangeformat{equation}{(#3#1#4--#5#2#6)}
\crefformat{enumi}{(#2#1#3)}
\crefrangeformat{enumi}{(#3#1#4--#5#2#6)}

\author[L.~Ferraro]{Luigi Ferraro}
\address{Department of Mathematics,
Texas Tech University, Lubbock, TX 79409, U.S.A.}
\email{lferraro@ttu.edu}

\author[D. Martin]{Desiree Martin}
\address{Department of Mathematics,
Syracuse University, Syracuse, NY 13244, U.S.A.}
\email{dmarti02@syr.edu}

\author[W. F. Moore]{W. Frank Moore}
\address{Department of Mathematics,
Wake Forest University, Winston-Salem, NC 27109, U.S.A.}
\email{moorewf@wfu.edu}

\begin{document}

\begin{abstract}
Let $\kk$ be a field and let $I$ be a monomial ideal in the polynomial ring $Q=\kk[x_1,\ldots,x_n]$. In her thesis, Taylor introduced a complex which provides a finite free resolution for $Q/I$ as a $Q$-module. Later, Gemeda constructed a differential graded structure on the Taylor resolution. More recently, Avramov showed that this differential graded algebra admits divided powers. We generalize each of these results to monomial ideals in a skew polynomial ring $R$. Under the hypothesis that the skew commuting parameters defining $R$ are roots of unity, we prove as an application that as $I$ varies among all ideals generated by a fixed number of monomials of degree at least two in $R$, there is only a finite number of possibilities for the Poincar\'{e} series of $\kk$ over $R/I$ and for the isomorphism classes of the homotopy Lie algebra of $R/I$ in cohomological degree larger or equal to two.
\end{abstract}

\maketitle

\section{Introduction}
Let $\kk$ be a field, let $n$ be a nonzero natural number and let $\mathfrak{q}=(q_{i,j})_{i,j=1,\ldots,n}$ be an $n\times n$ matrix with entries in $\kk^*$ such that $q_{i,j}=q_{j,i}^{-1}$ for all $i,j$ and $q_{i,i}=1$ for all $i$. By $R$ we denote the skew polynomial ring $\kk_\mathfrak{q}[x_1,\ldots,x_n]$, i.e. a polynomial ring where $x_ix_j=q_{i,j}x_jx_i$ for all $i,j$. In this paper we construct a finite free resolution 
\[
\mathbb{T}:0\longrightarrow T_s\longrightarrow T_{s-1}\longrightarrow\cdots\longrightarrow T_1\rightarrow R
\]
of cyclic modules of the form $R/I$ where $I$ is a monomial ideal of $R$ minimally generated by $s$ monomials. When the ring $R$ is commutative, i.e. when $q_{i,j}=1$ for all $i,j$, this resolution coincides with the well-known Taylor resolution constructed in \cite{Taylor}. As the Taylor resolution in the commutative case was a generalization of the Koszul complex, the resolution that we construct generalizes the skew Koszul complex constructed in \cite{FM}, which corresponds to the case where the monomials minimally generating $I$ form a regular sequence.

The study of multiplicative structures on finite free resolution of cyclic modules was initiated by Buchsbaum and Eisenbud in \cite{BE}, where they prove that a finite free resolution always admits a multiplicative structure associative only up to homotopy. In \cite{obstructions}, Avramov gave the first example of a cyclic module whose free resolution did not admit an associative multiplicative structure. Since the work of Buchsbaum and Eisenbud a lot of research has been devoted on the study of free resolutions with a multiplicative structure, see for example \cite{kustin} and the references therein.

In \cite{gemeda}, Gemeda showed that the Taylor resolution constructed in \cite{Taylor} admits a multiplicative structure, making it a differential graded algebra (DG algebra). In this paper we show that the free resolution $\mathbb{T}$ that we construct admits a structure of (noncommutative) DG algebra that coincides with the structure given by Gemeda in the commutative case.

In \cite{FM} it is shown that to a DG-algebra $A$ with divided powers (DG $\Gamma$-algebra) over the ring $R/I$ (or more generally over a quotient of $R$ by homogeneous normal elements), one can associate a (color) graded Lie algebra denoted by $\pi(A)$ and called \emph{homotopy Lie algebra} of $A$. The algebraic properties of $\pi(A)$ reflect the homological properties of the DG algebra $A$. We show that the DG algebra structure that we define on $\mathbb{T}$ admits divided powers, making $\mathbb{T}$ a DG $\Gamma$-algebra; this generalizes a result of Avramov \cite{AvramovTaylor}. Moreover we show that the homotopy Lie algebra of the ring $R/I$ in cohomological degree $\geq2$ is isomorphic to the Lie algebra of $\mathbb{T}\otimes_R\kk$. This gives a way to compute most of the homotopy Lie algebra of $R/I$ by studying the homotopy Lie algebra of the finite dimensional $\kk$-algebra $\mathbb{T}\otimes_R\kk$.

We use the connection between $\pi(R/I)$ and $\pi(\mathbb{T}\otimes_R\kk)$ to show that if the $q_{i,j}$ are roots of unity for all $i,j$ and if $I$ ranges among all the monomial ideals of $R$ minimally generated by a fixed number of monomials of degree at least two, then there are only a finite number of isomorphism classes for $\pi^{\geq2}(R/I)$, see \cref{cor:FiniteIsoClasses}, where a more general version of this statement is given. We provide an example showing that the hypothesis on the $q_{i,j}$ is needed, see \cref{example}. The $\kk$-vector space dimension of the graded components of $\pi(R/I)$ is connected to the Poincar\'{e} series of $\kk$ over $R/I$, i.e. the generating function of the Betti numbers of $\kk$ over $R/I$. As a result we deduce that if the $q_{i,j}$ are roots of unity for all $i,j$ and if $I$ ranges among all ideals of $R$ minimally generated by a fixed number of monomials of degree at least two, then there are only finitely many possibilities for the Poincar\'{e} series of $\kk$ over $R/I$. 
These last two statements are generalizations of results due to Avramov, that can be found in \cite{AvramovTaylor}.

The paper is organized as follows. In \cref{sec:Bicharacters} we construct two bicharacters associated to the ring $R$ and study their properties. These bicharacters are used in \cref{sec:TaylorRes} to construct a resolution $\mathbb{T}$ of $R/I$ over $R$ and their properties are used in proofs throughout the paper. In \cref{sec:DGA} we show that the resolution constructed in \cref{sec:TaylorRes} admits a structure of a (noncommutative) differential graded algebra and in \cref{sec:Gamma} we show that this DG algebra admits divided powers. In \cref{sec:pi} we show that the homotopy Lie algebra of $R/I$ in cohomological degree $\geq2$ is isomorphic to the homotopy Lie algebra of $\mathbb{T}\otimes_R\kk$. Finally, \cref{sec:Applications} is devoted to applications, in particular we state and prove the results on the finiteness of the isomorphism classes of the homotopy Lie algebras and of the Poincar\'{e} series of $\kk$ that were stated above.

\section{Bicharacters}\label{sec:Bicharacters}

Let $\kk$ be a commutative ring. Let $n$ be a positive integer and let $\mathfrak{q}=(q_{i,j})$ be a $n \times n$ matrix with entries in $\Bbbk^{*}$ such that $q_{i,i}=1$ for all $i=1,..,n$ and $q_{i,j}=q^{-1}_{j,i}$ for all $i,j=1,...,n$.
The \emph{skew polynomial ring} associated to the matrix $\mathfrak{q}$, denoted by $R=\kk_\mathfrak{q}[x_1,\ldots,x_n]$, is the Ore extension of $\kk$ with $n$ variables $x_1,\ldots,x_n$, such that $\kk$ is central in $R$ and $x_ix_j=q_{i,j}x_jx_i$ for all $i,j=1,\ldots,n$.

We say that a homogeneous element $f\in R$ is \textit{normal} if there exists a graded ring automorphism $\sigma$ such that $fg=\sigma(g)f$ for all homogeneous elements $g \in R$. We call $\sigma$ the \textit{normalizing automorphism} of $f$.

\begin{remark}
Let $G=\langle \sigma_1,\ldots,\sigma_n \rangle$ be the subgroup of the group of graded ring automorphisms of $R$, where $\sigma_i$ is the normalizing automorphism of $x_i$ for each $i=1, \ldots ,n$. We define a $\Z \times G$ grading on $R$. The $\Z$-degree of a monomial $x_1^{\alpha_1}\cdots x_n^{\alpha_n}$ is $ \alpha_1d_1+\cdots+\alpha_nd_n$ where $\mathrm{deg}(x_i)=d_i \in \Z^{+}$, and the $G$-degree of $x_1^{\alpha_1}\cdots x_n^{\alpha_n}$ is $\sigma_1^{\alpha_1} \cdots \sigma_n^{\alpha_n} \in G$. It is useful to notice that the $G$-degree of a monomial is exactly the normalizing automorphism of that monomial, i.e. $x_1^{\alpha_1} \cdots x_n^{\alpha_n}f= \sigma_1^{\alpha_1} \cdots \sigma_n^{\alpha_n}(f) x_1^{\alpha_1} \cdots x_n^{\alpha_n}$ for all $f\in R$. We call $G$ the \emph{group of colors} of $R$.

If $G$ is an abelian group, then an \emph{alternating bicharacter} on $G$ is a function $\chi:G\times G\rightarrow\kk^*$ such that for all $\alpha,\beta,\sigma\in G$
\begin{align*}
\chi(\sigma,\alpha\beta)&=\chi(\sigma,\alpha)\chi(\sigma,\beta),\\
\chi(\alpha\beta,\sigma)&=\chi(\alpha,\sigma)\chi(\beta,\sigma),\\
\chi(\sigma,\sigma)&=1.
\end{align*}

Let $G=\langle\sigma_1,\ldots, \sigma_n\rangle$ be the group of colors of a skew polynomial ring $R=\kk_\mathfrak{q}[x_1,\ldots,x_n]$. We define a map $$\chi: G \times G \rightarrow \kk^{*}$$ by extending the assignment $(\sigma_i, \sigma_j) \mapsto q_{i,j} $ to all of $G$ so that $\chi$ is a bicharacter. If $m,n \in R$ are $G$-homogeneous elements of $G$-degree $\sigma$ and $\tau$ respectively, we abuse notation and write $\chi(m,n)$ for $\chi(\sigma, \tau)$. Notice that if $m,n \in R$ are $G$-homogeneous then $mn=\chi(m,n)nm$. These are examples of \emph{color commutative rings}, see \cite[Definition 3.3]{FM} for the definition.
\end{remark}

For the remainder of the paper $R$ will denote a skew polynomial ring $\kk_\mathfrak{q}[x_1,\ldots,x_n]$, $G$ its group of colors, and $\chi$ the bicharacter on $G$ defined in the previous Remark.

All $R$-modules in this paper are right \emph{color $R$-modules} where a color $R$-module is an $R$-module with a compatible $\mathbb{Z}\times G$-grading. Note that a color right module $M$ can be turned into a \emph{color bimodule} by letting $rm= \chi(r,m)mr$ for $G$-homogeneous elements $r\in R$ and $m\in M$ .

\begin{definition}
 Let $ \mathbf{x^a}=x_{1}^{a_1}x_{2}^{a_2} \cdots x_{n}^{a_n}$ and $ \mathbf{x^b}=x_{1}^{b_1}x_{2}^{b_2} \cdots x_{n}^{b_n}$ be monomials in $R$ with $\mathbf{a}=(a_1, \ldots ,a_n) \in \N^{n}$. We say that $\mathbf{x^b}$ \textit{divides} $\mathbf{x^a}$, and write $\mathbf{x^b} | \mathbf{x^a}$, if $b_i \leq a_i$ for all $i=1,\ldots,n$. If $\mathbf{x^b} | \mathbf{x^a}$, then we define $\frac{\mathbf{x^a}}{\mathbf{x^b}}$ to be the monomial $\mathbf{x^{a-b}}$, where $\mathbf{a-b}=(a_1-b_1, \ldots, a_n-b_n) \in \N^{n}$.
\end{definition}

\begin{definition}
Let $\mathbf{x^a}=x_{1}^{a_1}x_{2}^{a_2} \cdots x_{n}^{a_n}$ and $\mathbf{x^b}=x_{1}^{b_1}x_{2}^{b_2} \cdots x_{n}^{b_n}$ be monomials in $R$. We define $\mathbf{x^{a}*x^{b}}$ to be the monomial $\mathbf{x^{a+b}}$, where $\mathbf{a+b}=(a_1+b_1, \ldots, a_n+b_n) \in \N^{n}$.

\end{definition}

\begin{remark}
If $X$ is the set of monomials of $R$, then $(X,*)$ is a monoid.
\end{remark}

\begin{definition}
 Let $\mathbf{x^a}$ and $ \mathbf{x^b}$ be monomials in $R$. We define $C(\mathbf{x^a},\mathbf{x^b})$ to be the constant such that $\mathbf{x^a x^b}= C(\mathbf{x^a}, \mathbf{x^b})\mathbf{x^{a}*x^{b}}$.
\end{definition}

\begin{lemma}
Let $\mathbf{x}^{\boldsymbol{\alpha}}$ and $\mathbf{x}^{\boldsymbol{\beta}}$ be monomials in $R$, with $\mathbf{\alpha}=(\alpha_1,\ldots,\alpha_n)$ and $\mathbf{\beta}=(\beta_1,\ldots,\beta_n)$. Then the following statements hold

\begin{enumerate}[label=(\alph*)]
    \item
$C(\mathbf{x^{\boldsymbol{\alpha}}},\mathbf{x^{\boldsymbol{\beta}}})=\displaystyle\prod_{i>j}\chi(x_i,x_j)^{\alpha_{i}\beta_{j}}$.
    \item
$C : X \times X \rightarrow \kk^{*}$ is a bicharacter on the monoid $X$.
    \item
$\frac{C \left(\mathbf{x}^{\boldsymbol{\alpha}},\mathbf{x}^{\boldsymbol{\beta}}\right)}{C \left( \mathbf{x}^{\boldsymbol{\beta}}, \mathbf{x}^{\boldsymbol{\alpha}} \right)}= \chi\left( \mathbf{x}^{\boldsymbol{\alpha}}, \mathbf{x}^{\boldsymbol{\beta}}\right)$.
\end{enumerate}
\end{lemma}

\begin{proof}
$\hphantom{a}$
\begin{enumerate}[label=(\alph*)]
\item We first notice that if $j<n$, then
\[
x_n^{\alpha_n}x_j^{\beta_j}=\chi(x_n,x_j)^{\alpha_n\beta_j}x_j^{\beta_j}x_n^{\alpha_n},
\]
thus 
\[\mathbf{x^{\boldsymbol{\alpha}}x^{\boldsymbol{\beta}}} = \displaystyle\prod_{n>j}\chi(x_n,x_j)^{\alpha_n\beta_j} x_1^{\alpha_1} \cdots x_{n-1}^{\alpha_{n-1}}x_1^{\beta_1} \cdots x_{n-1}^{\beta_{n-1}}x_{n}^{\alpha_n+\beta_n}.
\]
More generally if $j<i$, then
\[
x_i^{\alpha_i}x_j^{\beta_j}=\chi(x_i,x_j)^{\alpha_i\beta_j}x_j^{\beta_j}x_i^{\alpha_i},
\]
thus
\[
x_i^{\alpha_i}x_1^{\beta_1}\cdots x_i^{\beta_i}=\prod_{i>j}\chi(x_i,x_j)^{\alpha_{i}\beta_{j}}x_1^{\beta_1}\cdots x_i^{\alpha_i+\beta_i}.
\]
Therefore inductively one gets
\[
\mathbf{x^{\boldsymbol{\alpha}} x^{\boldsymbol{\beta}}}= \prod_{i>j}\chi(x_i,x_j)^{\alpha_{i}\beta_{j}} x_1^{\alpha_1+\beta_1} \cdots x_n^{\alpha_n+\beta_n},
\]
which by definition implies $C(\mathbf{x^{\boldsymbol{\alpha}}},\mathbf{x^{\boldsymbol{\beta}}})= \displaystyle\prod_{i>j}\chi(x_i,x_j)^{\alpha_{i}\beta_{j}}$.

    \item
Let $\mathbf{x^{\boldsymbol{\alpha}}}, \mathbf{x^{\boldsymbol{\beta}}}, \mathbf{x^{\boldsymbol{\gamma}}}$ be monomials in $R$. Then $$C(\mathbf{x^{\boldsymbol{\alpha}}}, \mathbf{x^{\boldsymbol{\beta}}})C(\mathbf{x^{\boldsymbol{\alpha}}},\mathbf{x^{\boldsymbol{\gamma}}})= \displaystyle\prod_{i>j}\chi(x_i,x_j)^{\alpha_{i}\beta_{j}} \displaystyle\prod_{i>j}\chi(x_{i},x_{j})^{\alpha_{i}\gamma_{j}}= \displaystyle\prod_{i>j}\chi(x_i,x_j)^{\alpha_{i}(\beta_{j}+\gamma_{j})}=C(\mathbf{x^{\boldsymbol{\alpha}}},\mathbf{x^{\boldsymbol{\beta+\gamma}}}),$$
and
$$C(\mathbf{x^{\boldsymbol{\alpha}}}, \mathbf{x^{\boldsymbol{\beta}}})C(\mathbf{x^{\boldsymbol{\gamma}}},\mathbf{x^{\boldsymbol{\beta}}})= \displaystyle\prod_{i>j}\chi(x_i,x_j)^{\alpha_{i}\beta_{j}} \displaystyle\prod_{i>j}\chi(x_{i},x_{j})^{\gamma_{i}\beta_{j}}= \displaystyle\prod_{i>j}\chi(x_i,x_j)^{(\alpha_{i}+ \gamma_{i})\beta_{j}}=C(\mathbf{x^{\boldsymbol{\alpha+\gamma}}},\mathbf{x^{\boldsymbol{\beta}}}).$$
Thus, $C$ is a bicharacter of the monoid $X$.
\item
By definition $\mathbf{x}^{\boldsymbol{\alpha}}\mathbf{x}^{\boldsymbol{\beta}}= C\left( \mathbf{x}^{\boldsymbol{\alpha}},\mathbf{x}^{\boldsymbol{\beta}}\right)\mathbf{x}^{\boldsymbol{\alpha}+\boldsymbol{\beta}}$. However,
\begin{align*}
\mathbf{x}^{\boldsymbol{\alpha}}\mathbf{x}^{\boldsymbol{\beta}}&=\chi\left(\mathbf{x}^{\boldsymbol{\alpha}},\mathbf{x}^{\boldsymbol{\beta}}\right)\mathbf{x}^{\boldsymbol{\beta}}\mathbf{x}^{\boldsymbol{\alpha}}\\
&=\chi\left(\mathbf{x}^{\boldsymbol{\alpha}},\mathbf{x}^{\boldsymbol{\beta}}\right)C\left( \mathbf{x}^{\boldsymbol{\beta}},\mathbf{x}^{\boldsymbol{\alpha}}\right)\mathbf{x}^{\boldsymbol{\alpha}+\boldsymbol{\beta}}.\qedhere\\
\end{align*}
\end{enumerate}
\end{proof}

\begin{remark}
We extend $C$ to the group of monomials in $x^{\pm1}_{1}, \ldots, x_{n}^{\pm1}$ by $$C(\mathbf{x^{\boldsymbol{\alpha-\beta}}},\mathbf{x^{\boldsymbol{\gamma-\delta}}})=C(\mathbf{x^{\boldsymbol{\alpha}}},\mathbf{x^{\boldsymbol{\gamma}}})C(\mathbf{x^{\boldsymbol{\alpha}}},\mathbf{x^{\boldsymbol{\delta}}})^{-1}C(\mathbf{x^{\boldsymbol{\beta}}},\mathbf{x^{\boldsymbol{\gamma}}})^{-1}C(\mathbf{x^{\boldsymbol{\beta}}},\mathbf{x^{\boldsymbol{\delta}}}).$$
\end{remark}

\begin{example}
In this example we show that the bicharacter $C$ cannot be defined as a bicharacter on the group of colors $G$. Indeed, let $q\in\kk^*$ be such that $q^2\neq1$, and let $R=\kk_\mathfrak{q}[x,y,z]$ be the skew polynomial ring such that $xy=qyx,xz=-zx,yz=-q^{-1}zy$. An easy check shows that the $G$-degrees of $x^2$ and $z^2$ are the same, but
\[
C(x^2,yz)=1,\quad\mathrm{and}\quad C(z^2,yz)=q^2.
\]
\end{example}

\begin{definition}
Let $\mathbf{x^{a _1}, \cdots, x^{a _m}}$ be monomials in $R$, where $\mathbf{a_i}=(a_{1,i} \ldots, a_{n,i})$. We define the \textit{least common multiple} of $\mathbf{x^{a_1}, \cdots, x^{a_m}}$, or $\mathrm{lcm}(\mathbf{x^{a_1}, \cdots, x^{a_m}})$, to be equal to $\mathbf{x}^{\mathrm{max} \; \mathbf{a_i}}$ where $\mathrm{max} \; \mathbf{a_i}=(\mathrm{max}_i \; a_{1,i} \;, \ldots,\; \mathrm{max}_i \; a_{n,i})$. If $m_1,\ldots,m_s$ are monomials and $F\subseteq[s]$, then $m_F$ denotes the least common multiple of the monomials $\{m_i\mid i\in F\}$.
\end{definition}

\section{The Taylor Resolution}\label{sec:TaylorRes}

\begin{construction}\label{cons:Taylor}
Let $m_1,\ldots,m_s$ be monomials in $R$. Let $T_1$ be the free $R$-module of rank $s$, and let $e_1,\ldots,e_s$ be its canonical basis. Let $T_i=\bigwedge^{i} T_1 \cong T_1^{\binom{s}{i}}$ for $i=0, \ldots, s$. A basis of $T_1^{\binom{s}{i}}$ is given by $\{e_F\mid |F|=i,F\subseteq[s]\}$.\\
We define maps $\partial_i:T_i\rightarrow T_{i-1}$ for $i=1,\ldots,s$ by $$ \partial_i (e_F) = \sum_{i \in F} e_{F \backslash\{i\}} (-1)^{\sigma(F,i)} C\left(m_{F \backslash\{i\}}, \frac{m_F}{m_{F \backslash\{i\}}}\right)^{-1} \frac{m_F}{m_{F \backslash\{i\}}}, $$
where $\sigma(F,i)= |\{j \in F : j< i\}|$.
In order for these maps to be homogeneous we set the $G$-degree and the internal degree of $e_{F}$ equal to the $G$-degree and the internal degree of $m_{F}$ respectively. 
\end{construction}

\begin{theorem}
Let $I$ be an ideal of $R$ generated by monomials $m_1,\ldots,m_s$, then $\mathbb{T}=(T_\bullet,\partial_\bullet)$ as defined above is a $\mathbb{Z}\times G$-graded complex of free $R$-modules.
\end{theorem}
\begin{proof}
In order to prove $\mathbb{T}$ is a complex, we show $\partial \circ \partial =0$. It follows from the definition of the differential that
$$\partial^2 (e_F)= \sum_{i \in F} \sum_{j \in F\backslash{ \{i \}}} e_{F \backslash{ \{i,j\}}} (-1)^{\sigma(F,i) + \sigma(F\backslash\{i\},j)} \widetilde{C}_{i,j} \frac{m_F}{m_{F \backslash{\{i,j\}}}},$$ where

$$\widetilde{C}_{i,j}= C\left(m_{F\backslash{\{i,j\}}}, \frac{m_{F\backslash{\{i\}}}}{m_{F\backslash{\{i,j\}}}}\right)^{-1} C\left(m_{F\backslash{\{i\}}}, \frac{m_{F}}{m_{F\backslash{\{i\}}}}\right)^{-1} C\left(\frac{m_{F\backslash{\{i\}}}}{m_{F\backslash{\{i,j\}}}}, \frac{m_{F}}{m_{F\backslash{\{i\}}}}\right).$$
By the commutative case it suffices to show $\widetilde{C}_{i,j}=\widetilde{C}_{j,i}$. To show the equality $\widetilde{C}_{i,j}=\widetilde{C}_{j,i}$ we prove the following\\
Claim: if $\mathbf{x^{\boldsymbol{\alpha}}}, \mathbf{ x^{\boldsymbol{\beta}}}, \mathbf{ x^{\boldsymbol{\gamma}}}, \mathbf{x^{\boldsymbol{\delta}}}$ are monomials such that $\mathbf{x^{\boldsymbol{\delta}}} | \mathbf{x^{\boldsymbol{\beta}}}$, $\mathbf{x^{\boldsymbol{\delta}}} | \mathbf{x^{\boldsymbol{\gamma}}}$, $\mathbf{x^{\boldsymbol{\beta}}} | \mathbf{x^{\boldsymbol{\alpha}}}$, $\mathbf{x^{\boldsymbol{\gamma}}} | \mathbf{x^{\boldsymbol{\alpha}}}$, then $$C(\mathbf{x^{\boldsymbol{\beta}}}, \mathbf{x^{\boldsymbol{\alpha}-\boldsymbol{\beta}}})C(\mathbf{x^{\boldsymbol{\delta}}}, \mathbf{x^{\boldsymbol{\beta}-\boldsymbol{\delta}}})C(\mathbf{x^{\boldsymbol{\gamma}-\boldsymbol{\delta}}}, \mathbf{x^{\boldsymbol{\alpha}-\boldsymbol{\gamma}}}) = C(\mathbf{x^{\boldsymbol{\gamma}}}, \mathbf{x^{\boldsymbol{\alpha}-\boldsymbol{\gamma}}})C(\mathbf{x^{\boldsymbol{\delta}}}, \mathbf{x^{\boldsymbol{\gamma}-\boldsymbol{\delta}}})C(\mathbf{x^{\boldsymbol{\beta}-\boldsymbol{\delta}}}, \mathbf{x^{\boldsymbol{\alpha}-\boldsymbol{\beta}}}).$$
To prove the claim it suffices to show the following equality\\ $$\frac{C(\mathbf{x^{\boldsymbol{\beta}}},\mathbf{x^{\boldsymbol{\alpha}-\boldsymbol{\beta}}})}{C(\mathbf{x^{\boldsymbol{\beta-\delta}}}, \mathbf{x^{\boldsymbol{\alpha-\beta}}})} \frac{C(\mathbf{x^{\boldsymbol{\delta}}}, \mathbf{x^{\boldsymbol{\beta}-\boldsymbol{\delta}}})}{C(\mathbf{x^{\boldsymbol{\delta}}},\mathbf{x^{\boldsymbol{\gamma-\delta}}})} \frac{C(\mathbf{x^{\boldsymbol{\gamma}-\boldsymbol{\delta}}}, \mathbf{x^{\boldsymbol{\alpha}-\boldsymbol{\gamma}}})}{C(\mathbf{x^{\boldsymbol{\gamma}}},\mathbf{x^{\boldsymbol{\alpha-\gamma}}})}=1.$$
Indeed,
\begin{align*}
\frac{C(\mathbf{x^{\boldsymbol{\beta}}},\mathbf{x^{\boldsymbol{\alpha}-\boldsymbol{\beta}}})}{C(\mathbf{x^{\boldsymbol{\beta-\delta}}}, \mathbf{x^{\boldsymbol{\alpha-\beta}}})} \frac{C(\mathbf{x^{\boldsymbol{\delta}}}, \mathbf{x^{\boldsymbol{\beta}-\boldsymbol{\delta}}})}{C(\mathbf{x^{\boldsymbol{\delta}}},\mathbf{x^{\boldsymbol{\gamma-\delta}}})} \frac{C(\mathbf{x^{\boldsymbol{\gamma}-\boldsymbol{\delta}}}, \mathbf{x^{\boldsymbol{\alpha}-\boldsymbol{\gamma}}})}{C(\mathbf{x^{\boldsymbol{\gamma}}},\mathbf{x^{\boldsymbol{\alpha-\gamma}}})} &= C(\mathbf{x^{\boldsymbol{\delta}}},\mathbf{x^{\boldsymbol{\alpha-\beta}}})C(\mathbf{x^{\boldsymbol{\delta}}},\mathbf{x^{\boldsymbol{\beta-\gamma}}})C(\mathbf{x^{\boldsymbol{-\delta}}}, \mathbf{x^{\boldsymbol{\alpha-\gamma}}})\\
&= C(\mathbf{x^{\boldsymbol{\delta}}},\mathbf{x^{\boldsymbol{\alpha-\gamma}}})C(\mathbf{x^{\boldsymbol{\delta}}},\mathbf{x^{\boldsymbol{\alpha-\gamma}}})^{-1}\\
&=1.\qedhere\\
\end{align*}
\end{proof}

\begin{remark}\label{rem:FrankBook}
It is proved in \cite[Theorem 1.3.6]{FrankBook} that every monomial ideal in a commutative polynomial ring with coefficients in a commutative ring admits a unique set of minimal generators. This proof can be adapted to our context and it is therefore omitted. Moreover in \cite[Theorem 1.1.9]{FrankBook} it is proved that if a monomial belongs to a monomial ideal, then it is a multiple of a minimal generator of the monomial ideal. This proof can also be adapted to our context and it is therefore omitted. 
\end{remark}

If $m_1,\ldots,m_s$ is a sequence of monomials in $R$. The \emph{Taylor complex} of $m_1,\ldots,m_s$ is the complex $\mathbb{T}=(T_\bullet,\partial_\bullet)$ defined in Construction \ref{cons:Taylor}.

\begin{theorem}
Let $I$ be an ideal of $R$ generated by monomials $m_1,\ldots,m_s$, then $\mathbb{T}=(T_\bullet,\partial_\bullet)$ as defined above is a $\mathbb{Z}\times G$-graded free $R$-resolution of $R/I$.
\end{theorem}

\begin{proof}
We prove this by inducting on $s$, with $s=1$ being clear. Assume $s>1$ and assume that the statement holds true for the sequence $m_1, \ldots, m_{s-1}$. Let $\T'$ be equal to the Taylor complex of the sequence of monomials $m_1 \ldots, m_{s-1}$, then $\T'$ is a subcomplex of $\T$ where $$\T'= \bigwedge^{i}R^{s-1} \hooklongrightarrow \bigwedge^{i}R^{s}$$ $$e_{j_{1}} \wedge \cdots \wedge e_{j_{i}} \mapsto e_{j_{1}} \wedge \cdots \wedge e_{j_{i}}$$
with $j_{l}=1, \ldots, s-1$. Consider $\mathbb{G}= \T/\T'$, then $G_{0}=0$ and $G_{i}=T_{i}/T'_{i}$ for $i>0$ is a free module generated by the basis $e_{F \cup \{s\}}$ where $|F|=i-1$ and $F \subseteq [s-1]$. The differential on $\mathbb{G}$ is given by $$\partial(e_{F \cup \{s\}})= \sum_{i \in F} e_{F\backslash\{i\} \cup \{s\} } (-1)^{\sigma(F,i)} C\left(m_{F\backslash\{i\} \cup \{s\}}, \frac{m_{F \cup \{s\}}}{m_{F\backslash\{i\} \cup \{s\}}} \right)^{-1} \frac{m_{F \cup \{s\}}}{m_{F\backslash\{i\} \cup \{s\}}}.$$ 

We prove $\mathbb{G}$ is isomorphic to the Taylor complex shifted by $1$ on the sequence of monomials $v_1, \ldots, v_{s-1}$ where $v_i= \mathrm{lcm}(m_i,m_s)/m_s $ through the isomorphism $\varphi$ defined by
$$\varphi(e_{F \cup \{s\}})= e_{F}(-1)^{|F|}C \left( m_s, v_{F} \right)^{-1}.$$
The map $\varphi$ is clearly bijective, it remains to show that is is a chain map. Indeed, 
\begin{align*}
    \varphi \partial^{\mathbb{G}}(e_{F} \cup \{s\}) &= \varphi \left(\sum_{i \in F} e_{F\backslash \{i\} \cup \{s\}} (-1)^{\sigma(F,i)} C\left( m_{F\backslash \{i\} \cup \{s\}}, \frac{m_{F \cup \{s\}}}{m_{F\backslash \{i\} \cup \{s\}}} \right)^{-1} \frac{m_{F \cup \{s\}}}{m_{F\backslash \{i\} \cup \{s\}}} \right)\\
    &= \sum_{i \in F} e_{F\backslash \{i\}}(-1)^{\sigma(i,F)+|F\backslash \{i\}|}C \left( m_s, v_{F \backslash \{i\}} \right)^{-1}C\left( m_{F\backslash \{i\} \cup \{s\}}, \frac{m_{F \cup \{s\}}}{m_{F\backslash \{i\} \cup \{s\}}}\right)^{-1} \frac{m_{F \cup \{s\}}}{m_{F\backslash \{i\} \cup \{s\}}}\\  
\end{align*}
and, denoting by $\mathbb{V}$ the shift of the Taylor complex on $v_1,\ldots,v_{s-1}$,
\begin{align*}
    \partial^{\mathbb{V}}\varphi(e_{F \cup \{s\}}) &= \partial^{\mathbb{V}} \left( e_{F}(-1)^{|F|}C \left(m_s, v_{F} \right)^{-1} \right)\\
    &=\sum_{i \in F} e_{F\backslash \{i\}}(-1)^{\sigma(F,i)+ |F|+1}C\left(v_{F\backslash\{i\}}, \frac{v_F}{v_{F\backslash\{i\}}} \right)^{-1}C\left(m_s, v_{F} \right)^{-1} \frac{v_F}{v_{F\backslash \{i\}}}.\\
\end{align*}
Thus, in order to show the isomorphism commutes with the differential we need to show the monomials and constants match. We start by proving the monomials are the same. Let $m_i=\mathbf{x}^{\boldsymbol{\alpha}_{i}}$ then $$m_{F \cup \{s\}}= \mathbf{x}^{\mathrm{max}_{i \in F \cup\{s\}}\boldsymbol{\alpha}_{i}}.$$ Thus, $$\frac{m_{F \cup \{s\}}}{m_{s}}= \mathbf{x}^{\mathrm{max}_{i \in F \cup \{s\}}(\boldsymbol{\alpha}_{i}-\boldsymbol{\alpha}_{s})}.$$ Therefore
\begin{align*}
v_{i}= \frac{m_{\{i,s\}}}{m_{s}}&= \mathbf{x}^{\mathrm{max}(\boldsymbol{\alpha}_{i},\boldsymbol{\alpha}_{s})-\boldsymbol{\alpha}_s}\\
&= \mathbf{x}^{\mathrm{max}(\boldsymbol{\alpha}_{i}-\boldsymbol{\alpha}_{s},0)}.\\
\end{align*}
It follows that
\begin{align*}
v_{F}&= \mathbf{x}^{\mathrm{max}_{i \in F}(\boldsymbol{\alpha}_{i}-\boldsymbol{\alpha}_{s}, 0)}\\
&=\mathbf{x}^{\mathrm{max}_{i \in F \cup \{s\}}(\boldsymbol{\alpha}_{i}-\boldsymbol{\alpha}_{s})}.\\
\end{align*}
Thus, $v_{F}= \frac{m_{F \cup \{s\}}}{m_{s}}$ which implies that the monomials appearing in $\varphi\partial^{\mathbb{G}}(e_F\cup\{s\})$ and $\partial^{\mathbb{V}}\varphi(e_{F\cup\{s\}})$ are the same.
Next we show the constants match. Indeed, $$\frac{C \left(m_{s},\frac{m_{F \cup \{s\}}}{m_{s}}\right)}{C\left(m_{s},\frac{m_{F\backslash \{i\} \cup \{s\}}}{m_{s}}\right)} \frac{C \left( \frac{m_{F \{i\} \cup \{s\}}}{m_{s}}, \frac{m_{F \cup \{s\}}}{m_{F\backslash\{i\} \cup \{s\}}}\right)}{C\left( m_{F\backslash\{i\} \cup \{s\}}, \frac{m_{F \cup \{s\}}}{m_{F\backslash\{i\} \cup \{s\}}}\right)}= C\left( m_{s}, \frac{m_{F \cup \{s\}}}{m_{F\backslash\{i\} \cup \{s\}}}\right)C \left( m_{s}, \frac{m_{F \cup \{s\}}}{m_{F\backslash\{i\} \cup \{s\}}} \right)^{-1}=1.$$
Since $\mathbb{G}$ is isomorphic to the shift of the Taylor complex by $1$ on the sequence of monomials $v_1, \ldots, v_{s-1}$, by induction $H_{i}(\mathbb{G})=0$ when $i \neq 1$ and $H_{1}(\mathbb{G})=R/(v_{1}, \ldots, v_{s-1})$. The exact sequence induced in homology by the exact sequence $$0 \rightarrow \T' \rightarrow \T \rightarrow \mathbb{G} \rightarrow 0$$ yields $H_{i}(\T)=0$ when $i>1$, and there is an exact sequence  $$0 \rightarrow H_{1}(\T') \rightarrow H_{1}(\T) \rightarrow H_{1}(\mathbb{G}) \xrightarrow{\eth} H_{0}(\T') \rightarrow H_{0}(\T) \rightarrow 0.$$ To conclude the proof we prove that the connecting homomorphism $\eth: H_{1}(\mathbb{G}) \rightarrow H_{0}(\T')$ is injective. We use bars to denote residue classes modulo ideals. A computation shows $\eth(\Bar{r})=\overline{m_{s}r}$, therefore it suffices to prove that if $m_{s}r \in (m_1, \ldots, m_{s-1})$ then $r \in (v_1, \ldots, v_{s-1})$. Without loss of generality we may assume $r$ is a monomial. The claim then follows from the fact that if a monomial belongs to a monomial ideal it must be a multiple of a minimal generator of the monomial ideal, see Remark \ref{rem:FrankBook}. 
\end{proof}

\begin{definition}
The \emph{skew Taylor resolution} of $R/I$ where $I$ is a monomial ideal of $R$ is the resolution $\mathbb{T}=(T_\bullet,\partial_\bullet)$.
\end{definition}

We recall that if $M$ is a graded $R$-module, then $\beta^R_i(M)$ denotes the rank of the free module in cohomological degree $i$ in a minimal (graded) free $R$-resolution of $M$. 

\begin{corollary}
Let $I$ be a monomial ideal of $R$ minimally generated by $s$ elements, then
\[
\beta^R_i(R/I)\leq\binom{s}{i}.
\]
\end{corollary}

\section{A Color DG Structure on the Taylor Resolution}\label{sec:DGA}
Color differential graded (DG) algebras have been introduced in \cite[Definition 4.1]{FM} by the first and third author. They have been used to study the homological properties of quotients of skew polynomial rings by ideals generated by homogeneous normal elements. A DG algebra $A$ over the skew polynomial ring $R=\kk_\mathfrak{q}[\mathbf{x}]$ is a bigraded unital associative $\kk$-algebra with $R\subseteq A_0$ equipped with a graded $R$-linear differential $\partial$ of homological degree $-1$ such that $\partial^2=0$, and such that the Leibniz rule holds:
\[
\partial(ab)=\partial(a)b+(-1)^{|a|}a\partial(b),
\]
where $a$ and $b$ are bihomogeneous elements.

The DG algebra $A$ is said to be a color DG algebra if $A$ has a $G$-grading compatible with the bigrading of $A$ and such that $\partial$ is $G$-homogeneous of $G$-degree $1_G$ (the identity of $G$).

We say that a color DG algebra $A$ is graded color commutative if for all trihomogeneous elements $x,y\in A$ one has
\[
xy=(-1)^{|x||y|}\chi(x,y)yx,
\]
where $\chi(x,y)$ is the usual abuse of notation, and $x^2=0$ if the homological degree of $x$ is odd.
\begin{theorem}\label{thm:DG}

The Taylor complex has a color commutative differential graded algebra structure given by the product 
\[ e_{V}e_{W}=
\begin{dcases}
e_{V \cup W}(-1)^{\sigma(V,W)}\frac{m_{V}*m_{W}}{m_{V \cup W}}C\left( m_{V}, m_{W} \right) C\left( m_{V \cup W}, \frac{m_{V}*m_{W}}{m_{V \cup W}} \right)^{-1} & V\cap W = \emptyset\\
0 & V\cap W \neq \emptyset\\
\end{dcases}
\]
where $\sigma(V, W)= |\{(i,j) \in V \times W \mid j<i\}|$.
\end{theorem}

\begin{proof}
We start by showing the product satisfies the Leibniz rule. We focus on the case $V\cap W=\emptyset$, leaving the other case to the reader.
We first compute $\partial(e_Ve_W)$,
\begin{align*}
    \partial ( e_{V}e_{W})&=\partial \left( e_{V \cup W}(-1)^{\sigma(V,W)}\frac{m_{V}*m_{W}}{m_{V \cup W}}C\left( m_{V}, m_{W} \right) C\left( m_{V \cup W}, \frac{m_{V}*m_{W}}{m_{V \cup W}} \right)^{-1} \right) \\
    &=\sum_{i \in V \cup W}e_{V \cup W \backslash\ii}(-1)^{\sigma(V,W)+ \sigma(V\cup W, i)}\frac{m_{V \cup W}}{m_{V \cup W \backslash \ii}} \frac{m_{V}*m_{W}}{m_{V \cup W}}\widetilde{C}_{V,W,i}\\
    &=\sum_{i \in V \cup W}e_{V \cup W \backslash\ii}(-1)^{\sigma(V,W)+ \sigma(V\cup W, i)} \frac{m_{V}*m_{W}}{m_{V \cup W \backslash \ii}} C\lp \frac{m_{V \cup W}}{m_{V \cup W \backslash \ii}}, \frac{m_{V}*m_{W}}{m_{V \cup W}}\rp\widetilde{C}_{V,W,i},\\
\end{align*}
where $\widetilde{C}_{V,W,i}=C\lp m_{V \cup W \backslash \ii}, \frac{m_{V \cup W}}{m_{V \cup W \backslash \ii}}\rp^{-1} C \lp m_{V}, m_{W}\rp C\lp m_{V \cup W}, \frac{m_{V}*m_{W}}{m_{V \cup W}} \rp^{-1}$.\\ \indent Next we compute $\partial(e_{V})e_{W}+ (-1)^{|W|}e_{V}\partial(e_{W})=$
\begin{align*}
    &\sum_{i \in V}e_{V \backslash \ii}(-1)^{\sigma(V,i)}C \lp m_{V \backslash \ii}, \frac{m_V}{m_{V\backslash \ii}} \rp^{-1}\frac{m_V}{m_{V \backslash \ii}}e_W \\
    & \indent + e_{V}\sum_{i\in W}e_{W \backslash \ii}(-1)^{|W| + \sigma(W,i)}C \lp m_{W \backslash \ii}, 
\frac{m_W}{m_{W \backslash\ii}}\rp^{-1} \frac{m_W}{m_{W \backslash\ii}}\\
&= \sum_{i \in V}e_{V \backslash \ii}e_{W}(-1)^{\sigma(V,i)} \frac{\chi(\frac{m_V}{m_{V \backslash \ii}}, m_{W})}{C \lp m_{V \backslash \ii}, \frac{m_V}{m_{V\backslash \ii}} \rp} \frac{m_V}{m_{V \backslash \ii}} \\
& \indent +\sum_{i\in W}e_{V}e_{W \backslash \ii}(-1)^{|W| + \sigma(W,i)}C \lp m_{W \backslash \ii}, 
\frac{m_W}{m_{W \backslash\ii}}\rp^{-1} \frac{m_W}{m_{W \backslash\ii}}\\
&= \sum_{i\in V}e_{V \backslash \ii 
    \cup W}(-1)^{\sigma(V,i)} \frac{m_{V\backslash \ii}*m_{W}}{m_{V\backslash \ii \cup W}}\frac{m_V}{m_{V \backslash \ii}} \frac {C \lp m_{V \backslash \ii}, m_{W}\rp}{ C \lp m_{V \backslash \ii \cup W}, \frac{m_{V \backslash \ii}*m_{W}}{m_{V\backslash \ii \cup W}} \rp} \frac{\chi(\frac{m_V}{m_{V \backslash \ii}}, m_{W})}{C \lp m_{V \backslash \ii}, \frac{m_V}{m_{V\backslash \ii}} \rp}\\
& \indent + \sum_{i\in W}e_{W \backslash \ii 
\cup V}(-1)^{|W| + \sigma(W,i)} \frac{m_{W\backslash \ii}*m_{V}}{m_{W\backslash \ii \cup V}} \frac{m_W}{m_{W \backslash \ii}} \frac{C \lp m_{V}, m_{W \backslash \ii} \rp}{C \lp m_{W \backslash \ii}, \frac{m_W}{m_{W \backslash \ii}} \rp} C \lp m_{V \cup W \backslash \ii}, \frac{m_{V}*m_{W \backslash \ii}}{m_{V \cup W \backslash
\ii}} \rp ^{-1}\\
&= \sum_{i\in V}e_{V \backslash \ii 
\cup W}(-1)^{\sigma(V,i)} \frac{m_{V}*m_{W}}{m_{V\backslash \ii \cup W}} \frac{C \lp \frac{m_{V\backslash\ii}*m_{W}}{m_{V \backslash\ii \cup W}}, \frac{m_{V}}{m_{V \backslash \ii}}\rp}{C \lp m_{V \backslash \ii \cup W}, \frac{m_{V \backslash \ii}*m_{W}}{m_{V \backslash\ii \cup W}} \rp} \frac{C \lp \frac{m_{V}}{m_{V \backslash\ii}}, m_{W}\rp}{C\lp m_{W}, \frac{m_{V}}{m_{V\backslash \ii}}\rp} \frac{C\lp m_{V \backslash \ii}, m_{W} \rp}{C \lp m_{V \backslash \ii}, \frac{m_{V}}{m_{V\backslash \ii}}\rp}\\
& \indent + \sum_{i\in W}e_{W \backslash \ii 
\cup V}(-1)^{|W| + \sigma(W,i)} \frac{m_{V}*m_{W}}{m_{W \backslash \ii \cup V}} \frac{C \lp \frac{m_{W\backslash \ii}*m_{V}}{m_{W\backslash \ii \cup V}}, \frac{m_W}{m_{W \backslash \ii}} \rp}{C \lp m_{V \cup W \backslash \ii}, \frac{m_{V}*m_{W \backslash \ii}}{m_{V \cup W \backslash
\ii}} \rp} \frac{C \lp m_{V}, m_{W \backslash \ii} \rp}{C \lp m_{W \backslash \ii}, \frac{m_W}{m_{W \backslash \ii}} \rp}.\\
\end{align*}

Ignoring the signs and the monomials, which will be the same by the commutative case, it suffices to show that for $i \in V$ one has 
\[
C\lp \frac{m_{V \cup W}}{m_{V \cup W \backslash \ii}}, \frac{m_{V}*m_{W}}{m_{V \cup W}}\rp\widetilde{C}_{V,W,i}= \frac{C \lp \frac{m_{V\backslash\ii}*m_{W}}{m_{V \backslash\ii \cup W}}, \frac{m_{V}}{m_{V \backslash \ii}}\rp}{C \lp m_{V \backslash \ii \cup W}, \frac{m_{V \backslash \ii}*m_{W}}{m_{V \backslash\ii \cup W}} \rp} \frac{C \lp \frac{m_{V}}{m_{V \backslash\ii}}, m_{W}\rp}{C\lp m_{W}, \frac{m_{V}}{m_{V\backslash \ii}}\rp} \frac{C\lp m_{V \backslash \ii}, m_{W} \rp}{C \lp m_{V \backslash \ii}, \frac{m_{V}}{m_{V\backslash \ii}}\rp}
\]
and for $i \in W$, one has $$ C\lp \frac{m_{V \cup W}}{m_{V \cup W \backslash \ii}}, \frac{m_{V}*m_{W}}{m_{V \cup W}}\rp 
\widetilde{C}_{V,W,i}= \frac{C \lp \frac{m_{W\backslash \ii}*m_{V}}{m_{W\backslash \ii \cup V}} \frac{m_W}{m_{W \backslash \ii}} \rp}{C \lp m_{V \cup W \backslash \ii}, \frac{m_{V}*m_{W \backslash \ii}}{m_{V \cup W \backslash
\ii}} \rp} \frac{C \lp m_{V}, m_{W \backslash \ii} \rp}{C \lp m_{W \backslash \ii}, \frac{m_W}{m_{W \backslash \ii}} \rp}.$$

First we will show $C\lp \frac{m_{V \cup W}}{m_{V \cup W \backslash \ii}}, \frac{m_{V}*m_{W}}{m_{V \cup W}}\rp\widetilde{C}_{V,W,i}= \frac{C \lp \frac{m_{V\backslash\ii}*m_{W}}{m_{V \backslash\ii \cup W}}, \frac{m_{V}}{m_{V \backslash \ii}}\rp}{C \lp m_{V \backslash \ii \cup W}, \frac{m_{V \backslash \ii}*m_{W}}{m_{V \backslash\ii \cup W}} \rp} \frac{C \lp \frac{m_{V}}{m_{V \backslash\ii}}, m_{W}\rp}{C\lp m_{W}, \frac{m_{V}}{m_{V\backslash \ii}}\rp} \frac{C\lp m_{V \backslash \ii}, m_{W} \rp}{C \lp m_{V \backslash \ii}, \frac{m_{V}}{m_{V\backslash \ii}}\rp}$. Indeed, 
\begin{align*}
    C\lp \frac{m_{V \cup W}}{m_{V \cup W \backslash \ii}}, \frac{m_{V}*m_{W}}{m_{V \cup W}}\rp\widetilde{C}_{V,W,i} &= \frac{C \lp \frac{m_{V \cup W}}{m_{V \cup W \backslash \ii}}, \frac{m_{V}*m_{W}}{m_{V \cup W}}\rp}{C\lp m_{V \cup W},\frac{m_{V}*m_{W}}{m_{V \cup W}} \rp}\frac{C\lp m_{V}, m_{W}\rp}{C\lp m_{V \cup W \backslash\ii}, \frac{m_{V \cup W}}{m_{V \cup W \backslash\ii}} \rp}\\
    &= C \lp \frac{1}{ m_{V \cup W \backslash \ii}}, \frac{m_{V}*m_{W}}{m_{V \cup W}} \rp C \lp \frac{1}{m_{V \cup W \backslash \ii}}, \frac{m_{V \cup W}}{m_{V \cup W \backslash \ii}} \rp C\lp m_{V}, m_{W}\rp\\
    &= C\lp m_{V \cup W \backslash \ii}, \frac{m_{V}*m_{W}}{m_{V \cup W \backslash \ii}} \rp^{-1} C\lp m_{V}, m_{W} \rp,\\
\end{align*}
and
\begin{align*}
    & \frac{C \lp \frac{m_{V\backslash\ii}*m_{W}}{m_{V \backslash\ii \cup W}}, \frac{m_{V}}{m_{V \backslash \ii}}\rp}{C \lp m_{V \backslash \ii \cup W}, \frac{m_{V \backslash \ii}*m_{W}}{m_{V \backslash\ii \cup W}} \rp} \frac{C \lp \frac{m_{V}}{m_{V \backslash\ii}}, m_{W}\rp}{C\lp m_{W}, \frac{m_{V}}{m_{V\backslash \ii}}\rp} \frac{C\lp m_{V \backslash \ii}, m_{W} \rp}{C \lp m_{V \backslash \ii}, \frac{m_{V}}{m_{V\backslash \ii}}\rp}\\ 
    &= \frac{C \lp \frac{m_{V\backslash\ii}*m_{W}}{m_{V \backslash\ii \cup W}}, \frac{m_{V}}{m_{V \backslash \ii}}\rp}{C \lp m_{V \backslash \ii \cup W}, \frac{m_{V \backslash \ii}*m_{W}}{m_{V \backslash\ii \cup W}} \rp} \frac{C \lp m_{V}, m_{W} \rp}{C \lp m_{W}*m_{V \backslash \ii}, \frac{m_{V}}{m_{V \backslash \ii}} \rp}\\
    &= \frac{C \lp m_{V \backslash \ii \cup W}, \frac{m_{V}}{m_{V \backslash \ii}} \rp^{-1}}{C \lp m_{V\backslash \ii \cup W}, \frac{m_{V \backslash \ii}*m_{W}}{m_{V \backslash \ii \cup W}} \rp} C\lp m_{V}, m_{W}\rp\\
    &= C\lp m_{V \cup W \backslash \ii}, \frac{m_{V}*m_{W}}{m_{V \cup W \backslash \ii}} \rp^{-1} C\lp m_{V}, m_{W} \rp.\\
\end{align*}
Next, we show $ C\lp \frac{m_{V \cup W}}{m_{V \cup W \backslash \ii}}, \frac{m_{V}*m_{W}}{m_{V \cup W}}\rp
\widetilde{C}_{V,W,i}= \frac{C \lp \frac{m_{W\backslash \ii}*m_{V}}{m_{W\backslash \ii \cup V}} \frac{m_W}{m_{W \backslash \ii}} \rp}{C \lp m_{V \cup W \backslash \ii}, \frac{m_{V}*m_{W \backslash \ii}}{m_{V \cup W \backslash
\ii}} \rp} \frac{C \lp m_{v}, m_{W \backslash \ii} \rp}{C \lp m_{W \backslash \ii}, \frac{m_W}{m_{W \backslash \ii}} \rp}$. Indeed,
\begin{align*}
    &\frac{C \lp \frac{m_{V}*m_{W \backslash \ii}}{m_{V \cup W \backslash \ii}}, \frac{m_{W}}{m_{W \backslash \ii}} \rp}{C \lp m_{W \backslash \ii}, \frac{m_{W}}{m_{W \backslash \ii}} \rp} \frac{C \lp m_{V}, m_{W \backslash \ii}\rp}{C \lp m_{V \cup W \backslash \ii}, \frac{m_{V}*m_{W \backslash\ii}}{m_{V \cup W \backslash \ii}} \rp}\\
    & \indent =  C \lp \frac{m_{V \cup W \backslash \ii}}{m_{V}}, \frac{m_{W}}{m_{W \backslash \ii}}\rp^{-1} \frac{C \lp m_{V}, m_{W \backslash \ii} \rp}{C \lp m_{V \cup W \backslash \ii}, \frac{m_{V}*m_{W \backslash \ii}}{m_{V \cup W \backslash \ii}}\rp}\\
    & \indent = C\lp m_{V}, m_{W} \rp C \lp m_{V}, m_{W \backslash \ii} \rp^{-1} C\lp m_{V \cup W \backslash \ii}, \frac{m_{W}}{m_{W \backslash \ii}} \rp^{-1} \frac{C\lp m_{V}, m_{W \backslash\ii}\rp}{C \lp m_{V \cup W \backslash \ii}, \frac{m_{V}*m_{W \backslash\ii}}{m_{V \cup W \backslash \ii}} \rp}\\
    & \indent = C\lp m_{V}, m_{W} \rp C\lp m_{V \cup W \backslash \ii}, \frac{m_{V}*m_{W}}{m_{V \cup W \backslash \ii}} \rp^{-1}.
\end{align*}
Thus, the Leibniz rule holds for this product.\\
\indent Now we prove the associativity of this product. In the following computations we will be ignoring the signs and the monomials, which will be the same by the commutative case. For the product $\lp e_{V}e_{W}\rp e_{X}$ we have constants,
\begin{align*}
   & C\lp m_{V \cup W}, m_{X}\rp \chi\lp\frac{m_{V}*m_{W}}{m_{V \cup W}}, m_{X} \rp \frac{C\lp m_{V},m_{W} \rp}{C \lp m_{V \cup W \cup X}, \frac{m_{V \cup W}*m_{X}}{m_{V \cup W \cup X}} \rp} \frac{C\lp \frac{m_{V \cup W}*m_{X}}{m_{V \cup W \cup X}}, \frac{m_{V}*m_{W}}{m_{V \cup W}} \rp}{C\lp m_{V \cup W}, \frac{m_{V}*m_{W}}{m_{V \cup W}} \rp}\\
   &= C\lp m_{V \cup W}, m_{X}\rp \frac{C\lp \frac{m_{V}*m_{W}}{m_{V \cup W}}, m_{X}\rp}{C\lp m_{X}, \frac{m_{V}*m_{W}}{m_{V \cup W}} \rp} \frac{C\lp m_{V},m_{W} \rp}{C \lp m_{V \cup W \cup X}, \frac{m_{V \cup W}*m_{X}}{m_{V \cup W \cup X}} \rp} C\lp \frac{m_{X}}{m_{V \cup W \cup X}},\frac{m_{V}*m_{W}}{m_{V \cup W}} \rp\\
   &= C\lp m_{V}*m_{W}, m_{X}\rp C\lp \frac{1}{m_{X}}, \frac{m_{V}*m_{W}}{m_{V \cup W}} \rp \frac{C\lp m_{V},m_{W} \rp}{C \lp m_{V \cup W \cup X}, \frac{m_{V \cup W}*m_{X}}{m_{V \cup W \cup X}} \rp} C\lp \frac{m_{X}}{m_{V \cup W \cup X}},\frac{m_{V}*m_{W}}{m_{V \cup W}} \rp\\
   &=  C\lp m_{V}*m_{W}, m_{X}\rp C\lp m_{V}, m_{W}\rp C\lp m_{V \cup W \cup X}, \frac{m_{V \cup W}*m_{X}}{m_{V \cup W \cup X}} \rp^{-1} C\lp m_{V \cup W \cup X}, \frac{m_{V}*m_{W}}{m_{V \cup W}} \rp^{-1}\\
   &= C\lp m_{V}*m_{W}, m_{X}\rp C\lp m_{V}, m_{W}\rp C\lp m_{V \cup W \cup X}, \frac{m_{V}*m_{W}*m_{X}}{m_{V \cup W \cup X}} \rp^{-1}.\\
\end{align*}
Expanding the product $e_{V}\lp e_{W}e_{X}\rp$ we have constants,
\begin{align*}
   & C\lp m_{W}, m_{X} \rp \frac{C\lp m_{V}, m_{W\cup X} \rp}{C \lp m_{V\cup W\cup X}, \frac{m_{V}*m_{W\cup X}}{m_{V\cup W\cup X}} \rp} \frac{C\lp \frac{m_{V}*m_{W \cup X}}{m_{V\cup W\cup X}}, \frac{m_{W}*m_{X}}{m_{W\cup X}} \rp}{C\lp m_{W \cup X}, \frac{m_{W}*m_{X}}{m_{W\cup X}}\rp}\\
   & = C\lp m_{W}, m_{X} \rp \frac{C\lp m_{V}, m_{W\cup X} \rp}{C \lp m_{V\cup W\cup X}, \frac{m_{V}*m_{W\cup X}}{m_{V\cup W\cup X}} \rp} C\lp \frac{m_{V}}{m_{V\cup W\cup X}}, \frac{m_{W}*m_{X}}{m_{W\cup X}} \rp\\
   &=C\lp m_{W}, m_{X} \rp \frac{C\lp m_{V}, m_{W\cup X} \rp}{C \lp m_{V\cup W\cup X}, \frac{m_{V}*m_{W\cup X}}{m_{V\cup W\cup X}} \rp} \frac{C \lp m_{V}, \frac{m_{W}*m_{X}}{m_{W\cup X}}\rp}{C\lp m_{V\cup W\cup X}, \frac{m_{W}*m_{X}}{m_{W\cup X}} \rp}\\
   &= C\lp m_{W}, m_{X}\rp C\lp m_{V}, m_{W}*m_{X}\rp C\lp m_{V\cup W\cup X}, \frac{m_{V}*m_{W}*m_{X}}{m_{V\cup W\cup X}} \rp^{-1}\\
   &= C\lp m_{W}, m_{X}\rp C\lp m_{V}, m_{X} \rp C\lp m_{V}, m_{W} \rp C\lp m_{V\cup W\cup X}, \frac{m_{V}*m_{W}*m_{X}}{m_{V\cup W\cup X}} \rp^{-1}\\
   &= C\lp m_{W}* m_{V}, m_{X}\rp C\lp m_{V}, m_{W}\rp C\lp m_{V\cup W\cup X}, \frac{m_{V}*m_{W}*m_{X}}{m_{V\cup W\cup X}} \rp^{-1}.\\
\end{align*}
Thus, the constants are the same and therefore associativity holds.\\
\indent Finally, it remains to show that this product is graded color commutative, i.e. \begin{equation}\label{eq:ColorComm}
e_{V}e_{W}=(-1)^{|V| \cdot |W|} \chi \lp e_{V}, e_{W} \rp e_{W}e_{V}.
\end{equation}
Similar to the proofs for the Leibniz rule and associativity, by the commutative we can ignore the signs and the monomials. By expanding the right side of \eqref{eq:ColorComm} we obtain
\begin{align*}
    \chi\lp e_{V}, e_{W} \rp C\lp m_{W}, m_{V}\rp C\lp m_{V\cup W}, \frac{m_{W}*m_{V}}{m_{V\cup W}} \rp^{-1} &= \chi\lp m_{V}, m_{W} \rp C\lp m_{W}, m_{V}\rp C\lp m_{V\cup W}, \frac{m_{W}*m_{V}}{m_{V\cup W}} \rp^{-1}\\
    &= \frac{C\lp m_{V}, m_{W} \rp}{C\lp m_{W}, m_{V} \rp} C\lp m_{W}, m_{V} \rp C\lp m_{V\cup W}, \frac{m_{W}*m_{V}}{m_{V\cup W}} \rp^{-1}\\
    &= C\lp m_{V}, m_{W} \rp C\lp m_{V\cup W}, \frac{m_{W}*m_{V}}{m_{V\cup W}} \rp^{-1}\\
\end{align*}
which is the constant obtained by expanding the left side of \eqref{eq:ColorComm}.
\end{proof}

\section{Divided Powers}\label{sec:Gamma}
In this section we show that the Taylor resolution admits a structure of color DG algebra with divided powers (or color DG $\Gamma$-algebra). A color DG $R$ algebra $A$ is said to have divided powers if to every element $x\in A$ of even positive homological degree, there is associated a sequence of elements $x^{(k)}\in A$ ($k=0,1,2,\ldots$) satisfying a list of axioms that can be found in \cite[Definition 6.1]{FM}.

\begin{definition}
Let $\mathbf{x^{a _1}, \cdots, x^{a _m}}$ be monomials in $R$, where $\mathbf{a_i}=(a_{1,i} \ldots, a_{n,i})$. We define the \textit{greatest common divisor} of $\mathbf{x^{a_1}, \cdots, x^{a_m}}$, or $\mathrm{gcd}(\mathbf{x^{a_1}, \cdots, x^{a_m}})$, to be equal to $\mathbf{x}^{\mathrm{min} \; \mathbf{a_i}}$ where $\mathrm{min} \; \mathbf{a_i}=(\mathrm{min}_i \; a_{1,i} \;, \ldots,\; \mathrm{min}_i \; a_{n,i})$.
\end{definition}

\begin{remark}
It is straightforward to check that
\[
\mathrm{gcd}(\mathbf{x^{\alpha}, x^{\beta}})=\frac{\mathbf{x^{\alpha}}*\mathbf{x^{\beta}}}{\mathrm{lcm}(\mathbf{x^{\alpha}, x^{\beta}})},
\]
in particular the formula in Theorem \ref{thm:DG} can be written in terms of a gcd.

\end{remark}

\begin{remark}\label{rem:DGGamma}
If $U$ and $V$ are color DG $\Gamma$-algebras over $R$, then so is $U\otimes_RV$, this is a color version of \cite[1.8.3]{GulliksenLevin}. Every color DG algebra containing $\mathbb{Q}$ is a color DG $\Gamma$-algebra, this is a color version of \cite[1.7.2]{GulliksenLevin}.
\end{remark}

\begin{theorem}
Let $I=(m_1,\ldots,m_s)$ be a monomial ideal of $R$ minimally generated by $s$ elements, then the Taylor resolution $\mathbb{T}$ of $R/I$ admits a unique color DG $\Gamma$-structure, given by the following formula
\[
\left(\sum_{\substack{P\subseteq[s]\\|P|=m}}a_Pe_P\right)^{(r)}=\sum_{\substack{P_1,\ldots, P_r\subseteq[s]\\|P_i|=m,\;i=1,\ldots,r}}(a_{P_1}e_{P_1})\cdots(a_{P_r}e_{P_r}),\quad\forall m\in2\mathbb{N}\backslash\{0\},r\in\mathbb{N},
\]
where $\displaystyle\sum_{\substack{P\subseteq[s]\\|P|=m}}a_Pe_P$ is a homogeneous (with respect to all degrees) element of $\mathbb{T}$.
\end{theorem}

\begin{proof}
We first prove that there exists a color DG $\Gamma$-structure on $\mathbb{T}$.

Consider the following rings
\[
F^{\mathbb{Q}}=\mathbb{Q}(y_{i,j}\mid i<j; i,j=1,\ldots,n),\quad\mathrm{and}\quad F^{\mathbb{Z}}=\mathbb{Z}(y_{i,j}\mid i<j; i,j=1,\ldots,n),
\]
where the $y_{i,j}$'s are indeterminates. 
Consider the following Ore extensions $S=F^{\mathbb{Q}}[x_1,\ldots, x_n]$ and $A=F^{\mathbb{Z}}[x_1,\ldots,x_n]$, where $x_ix_j=y_{i,j}x_jx_i$ for all $i,j=1,\ldots,n$ while $\mathbb{Q}$ and $\mathbb{Z}$ are central in $S$ and $A$ respectively. Let $M_I$ be the minimal set of generators of $I$, then we denote by $IS$ and $IA$ the ideals of $S$ and $A$ generated by $M_I$. We denote by $\mathbb{T}_S$ and $\mathbb{T}_A$ the Taylor resolution of $S/IS$ over $S$ and of $A/IA$ over $A$ respectively. The DG algebra $\mathbb{T}_S$ is a DG $\Gamma$-algebra by Remark \ref{rem:DGGamma} since it contains $\mathbb{Q}$. The canonical map $\mathbb{T}_A\rightarrow F^\Q\otimes_{F^\Z}\mathbb{T}_A=\mathbb{T}_S$ induces a DG $\Gamma$-structure on $\mathbb{T}_S$ by Remark \ref{rem:DGGamma}. The ring $\kk$ is a $F^\Z$-module through the map $F^\Z\rightarrow\kk$ that sends $1\mapsto1$ and $y_{i,j}\mapsto q_{i,j}$ for all $i,j=1\ldots, n$. The canonical map $\mathbb{T}_A\rightarrow \kk\otimes_{F^\Z}\mathbb{T}_A=\mathbb{T}$ induces a DG $\Gamma$-structure on $T$ by Remark \ref{rem:DGGamma}.

Now we prove the uniqueness.

Let $P=\{i_1,\ldots,i_m\}\subseteq[s]$ with $m\in 2\mathbb{N}\backslash\{0\}$, we define
\[
d_P=\prod_{l=1}^{m-1}\mathrm{gcd}(m_{\{i_1,\ldots,i_l\}},m_{i_{l+1}}),
\]
then for $r\geq2$ one has
\[
d_P^re_P^{(r)}=\chi(d_P,e_P)^{\binom{r}{2}}(d_Pe_P)^{(r)}=\left(\pm c\prod_{i\in P}e_i\right)^{(r)}=(\pm c)^r\left(\prod_{i\in P}e_i\right)^{(r)}=0,
\]
for some $c\in\kk^*$, where the first, third and fourth equality follow from \cite[Definition 6.1(4)]{FM} and the second equality follows from the product formula in Theorem \ref{thm:DG}. Since $d_P$ is not a zero-divisor on $\mathbb{T}$, it follows that $e_P^{(r)}=0$ for $r\geq2$. The assertion of the Theorem follows from the following string of equalities
\begin{align*}
\left(\sum_{\substack{P\subseteq[s]\\|P|=m}}a_Pe_P\right)^{(r)}&=\sum_{\substack{P_1,\ldots, P_r\subseteq[s]\\|P_i|=m,\;i=1,\ldots,r\\q_1+\cdots+ q_r=r}}(a_{P_1}e_{P_1})^{(q_1)}\cdots(a_{P_r}e_{P_r})^{(q_r)}\\&=\sum_{\substack{P_1,\ldots, P_r\subseteq[s]\\|P_i|=m,\;i=1,\ldots,r\\q_1+\cdots+ q_r=r}}\left(\chi(e_{P_1},a_{P_1})^{\binom{q_1}{2}}a_{P_1}^{q_1}e_{P_1}^{(q_1)}\right)\cdots\left(\chi(e_{P_r},a_{P_r})^{\binom{q_r}{2}}a_{P_r}^{q_r}e_{P_r}^{(q_r)}\right)\\
&=\sum_{\substack{P_1,\ldots, P_r\subseteq[s]\\|P_i|=m,\;i=1,\ldots,r}}(a_{P_1}e_{P_1})\cdots(a_{P_r}e_{P_r}),
\end{align*}
where the first equality follows from \cite[Definition 6.1(3)]{FM}, the second from \cite[Definition 6.1(4)]{FM}, and the third follows from the fact that $e_P^{(r)}=0$ for $r\geq2$.

\end{proof}

\section{Color Homotopy Lie Algebra}\label{sec:pi}
In this section $\kk$ is assumed to be a field. Let $S$ be a color commutative ring over $\kk$.
\begin{chunk}
Let $A$ be a color DG $\Gamma$-algebra over $S$. By adjoining either exterior variables or divided powers variables to $A$, we can construct a minimal free $A$-resolution of $\kk$ with the structure of DG $\Gamma$-algebra, see \cite[Construction 2.4 and Construction 2.6]{FM}. The minimal resolution obtained in this manner is called an \emph{acyclic closure} of $\kk$ over $A$, see \cite[Definition 5.6]{FM}. We denote this acyclic closure by $A\langle X\rangle$, where $X$ is the set of variables that have been adjoined.
\end{chunk}
\begin{chunk}
Let $U$ be a color module over the graded algebra $A\langle X\rangle^\natural$ (where $^\natural$ denotes the forgetful functor from the category of DG $S$-algebras to the category of graded $S$-algebras). An $A$-\emph{linear color derivation} $D:A\langle X\rangle\rightarrow U$ is an element of $\mathrm{Hom}_A(A\langle X\rangle,U)$ such that
\begin{align*}
D(a)   =~& 0 & & \text{for all}~a \in A, \\
D(bb') =~& D(b)b' + (-1)^{|D||b|}\chi(D,b)bD(b') & & \text{for all}~b,b' \in A\langle X\rangle, b,b'~\text{trihomogeneous}, \\
D(x^{(i)}) =~& D(x)x^{(i-1)} & & \text{for all}~x \in X_\text{even}~\text{and all}~i \in \mathbb{N},
\end{align*}
We denote the set of $A$-linear color derivations from $A\langle X\rangle$ to $U$ by $\mathrm{Der}_A(A\langle X\rangle, U)$. One can check that if $U$ is a color DG $A$-module, then $\mathrm{Der}_A(A\langle X\rangle, U)$ is a color DG $A\langle X\rangle$-submodule of $\mathrm{Hom}_A(A\langle X\rangle,U)$.
\end{chunk}

\begin{chunk}
The same proof used in \cite[Proposition 7.10]{FM} shows that $\mathrm{Der}_A(A\langle X\rangle,A\langle X\rangle)$ is a color DG Lie algebra, see \cite[Definition 7.1]{FM}. We define the color homotopy Lie algebra of $A$ to be the homology of the color DG Lie algebra $\Der_A(A\langle X\rangle,A\langle X\rangle)$ and denote it by $\pi(A)$.

\end{chunk}

\begin{chunk}
It is shown in \cite[Proposition 5.2]{FM} that  there exists a (semi-free) color DG module $\Diff_AA\langle X\rangle$ over $A\langle X\rangle$
and a chain derivation $d :A\langle X\rangle \to \Diff_AA\langle X\rangle$ of trivial cohomological and $G$-degree
such that
\begin{enumerate}
\item The $A\langle X\rangle^\natural$-module $(\Diff_A(A\langle X\rangle)^\natural$ has a basis
$dX = \{dx~|~x \in X\}$, where $dx$ and $x$ have the same internal, homological,
and $G$-degrees.
\item $d(x) = dx$ for all $x \in X$.
\item $\partial(b(dx)) = \partial(b)(dx) + (-1)^{|b|}bd(\partial(x))$ for all
$b \in A\langle X\rangle$.
\item The map 
\begin{eqnarray*}
\Hom_{A\langle X\rangle}(\Diff_A A\langle X\rangle,U) & \to & \Der_A(A\langle X\rangle,U) \\
 \beta & \mapsto & \beta \circ d
\end{eqnarray*}
is an isomorphism (natural in $U$) of color DG $A\langle X\rangle$-modules
for every DG $A\langle X\rangle$-module $U$.
\end{enumerate}
\end{chunk}

Let $S=R/I$ with $I$ monomial ideal generated by monomials of degree at least 2, we denote by $K^S$ the skew Koszul complex on the images of the variables of $R$ in $S$, see \cite[Construction 2.10]{FM}. 

For the remainder of the paper $\pi^{\geq2}(S)$ denotes the homotopy Lie algebra of $S$ concentrated in cohomological degrees larger or equal to 2.
\begin{proposition}\label{prop:piKoszul}
Let $S=R/I$ with $I$ monomial ideal generated by monomials of degree at least 2. Then there is an isomorphism of graded color Lie algebras $\pi(K^S)\cong\pi^{\geq2}(S)$.
\end{proposition}

\begin{proof}
Let $F=S\langle X\rangle$ be the acyclic closure of $\kk$ over $S$, where $X$ is the set of variables adjoint, see \cite[Definition 5.6]{FM}. Then $F$ is also the acyclic closure of $\kk$ over $K^S$ since $F=K^S\langle X_{\geq2}\rangle$, where $X_{\geq2}$ is the set of variables of homological degree larger or equal to two. For the definition and properties of the module of differentials $\Diff$ see \cite[Proposition 5.2]{FM}. Consider the following commuting diagram

\begin{equation*}
\begin{tikzpicture}[baseline=(current  bounding  box.center)]
 \matrix (m) [matrix of math nodes,row sep=3em,column sep=4em,minimum width=2em] {
\Der_{K^S}(F,F)&\Der_S(F,F)\\
\Der_{K^S}(F,\kk)&\Der_S(F,\kk)\\
\Hom_F(\Diff_{K^S}F,\kk)&\Hom_F(\Diff_SF,\kk)\\
\Hom_{\kk}\lp \lp\Diff_{K^S}F\rp\otimes_F\kk,\kk\rp&\Hom_{\kk}\lp \lp\Diff_{S}F\rp\otimes_F\kk,\kk\rp\\
\Hom_{\kk}(\kk X_{\geq2},\kk)&\Hom_{\kk}(\kk X,\kk)\\};
\path[->] (m-1-1) edge (m-1-2);
\path[->] (m-5-1) edge (m-5-2);
\path[->] (m-1-1) edge (m-2-1);
\path[->] (m-1-2) edge (m-2-2);
\path[->] (m-3-1) edge node[left] {\cite[Proposition 5.2]{FM}} (m-2-1);
\path[->] (m-3-2) edge node[right] {\cite[Proposition 5.2]{FM}} (m-2-2);
\draw [double equal sign distance]  (m-3-1) to [out=270, in=90] (m-4-1);
\draw [double equal sign distance]  (m-3-2) to [out=270, in=90] (m-4-2);
\draw [double equal sign distance]  (m-4-1) to [out=270, in=90] (m-5-1);
\draw [double equal sign distance]  (m-4-2) to [out=270, in=90] (m-5-2);
\end{tikzpicture}
\end{equation*}
where the first horizontal map is induced by the inclusion $S\subseteq K^S$, and the last horizontal map is the obvious one. The vertical maps on the top right and left are induced by the augmentation $F\rightarrow\kk$, and therefore are quasi-isomorphisms by \cite[Corollary 5.2]{FM}. The assertion of the Proposition now follows by taking homology in this diagram.

\end{proof}

\begin{remark}\label{thm:piFunctorial}
Let $S$ be a color commutative ring. A \emph{local} color DG $S$-algebra $A$ is a homologically nonnegatively graded color DG $S$-algebra with a unique maximal DG ideal $\m_A$, such that $H_i(A)$ is finitely generated over $S$ for each $i\geq 0$, and $S\to A_0$ is a surjective local homormophism. In this case, $\m_A$ can be written as
\[
\m_A=\m_S\oplus A_1\oplus A_2\oplus \ldots.
\]
A morphism of local color DG $S$-algebras $\vp:A\to B$ is a morphism of color DG $S$-algebras satisfying $\vp(\m_A)\subseteq \m_B.$

Let $\vp:A\rightarrow B$ be a quasi-isomorphism of local color DG $S$-algebras. Then there is an induced map $\pi(\vp):\pi(B)\rightarrow\pi(A)$ which is an isomorphism of graded color Lie algebras.
This is a well-known result in the commutative case, for a proof in the color case see \cite{FMP}.
\end{remark}

The next Proposition gives a way to compute most of the color homotopy Lie algebra of $S$ by working with the Taylor resolution of $S$ over $R$.
\begin{proposition}\label{prop:pi>=2}
Let $I$ be a monomial ideal of $R$ minimally generated by monomials of degree at least 2, let $S=R/I$ and let $\mathbb{T}$ be the Taylor resolution of $I$ over $S$, then there is an isomorphism of graded color Lie algebras
\[
\pi^{\geq2}(S)\cong\pi(\mathbb{T}\otimes_R\kk).
\]
\end{proposition}
\begin{proof}
Let $\varepsilon:K^R\rightarrow\kk$ and $\tau:\mathbb{T}\rightarrow S$ be the augmentations of $K^R$ and $\mathbb{T}$. Consider the following maps, which are quasi-isomorphisms by \cite[Proposition 4.12]{FM},
\[
K^S=S\otimes_RK^R\xleftarrow{\;\;\tau\otimes_S K^R\;\;} \mathbb{T}\otimes_R K^R\xrightarrow{\;\; \mathbb{T}\otimes_R\varepsilon\;\;}\mathbb{T}\otimes_R\kk,
\]
where the equality on the left follows since $I$ is generated by monomials of degree at least 2. As a consequence, there is the following chain of isomorphisms of graded color Lie algebras
\begin{align*}
\pi^{\geq2}(S)&\cong \pi(K^S)\\
&\cong\pi(\mathbb{T}\otimes_RK^R)\\
&\cong\pi(\mathbb{T}\otimes_R\kk),
\end{align*}
where the the first isomorphism follows from Proposition \ref{prop:piKoszul} and the last two isomorphisms follow from Remark \ref{thm:piFunctorial}.
\end{proof}

\section{Lattices and Graphs}\label{sec:Applications}

In this section $\Bbbk$ is also assumed to be a field.

\begin{definition}
Let $I$ be a monomial ideal of $R$ minimally generated by the monomials $m_1,\ldots,m_s$. The \emph{LCM lattice} of $I$ is the set 
\[
L_I=\{m_F\mid F\subseteq[s]\},
\]
ordered by divisibility $m_P\leq m_Q$ if and only if $m_P\mid m_Q$; with the convention $m_\emptyset=1$. The join of $m_P$ and $m_Q$ is $m_{P\cup Q}$, while the meet is $m_{P\cap Q}$. 
\end{definition}

\begin{definition}
Let $I$ be a monomial ideal of $R$ minimally generated by the monomials $m_1,\ldots, m_s$. The \emph{GCD graph} of $I$, denoted by $G_I$, has vertices the elements of $L_I$. There is an edge from $m_P$ to $m_Q$ with weight $C(m_P,m_Q)$ and an edge from $m_Q$ to $m_P$ with weight $C(m_Q,m_P)$ if and only if $\mathrm{gcd}(m_P,m_Q)=1$.
\end{definition}

\begin{remark}
The Boolean lattice on $[s]$ is the set $\mathcal{P}([s])$ (set of subsets of $[s]$) ordered by inclusion. The join of two subsets is the union and the meet is the intersection. Let $I,I'$ be ideals of skew polynomial rings $R$ and $R'$ respectively, minimally generated by $s$ monomials. Let $\lambda:L_I\rightarrow L_{I'}$ be an isomorphism of lattices. The atoms of $L_I$ and $L_{I'}$ are the minimal generating set of $I$ and $I'$ respectively, therefore $\lambda$ maps the minimal generators of $I$ bijectively onto the minimal generators of $I'$. So $\lambda$ induces an isomorphism of lattices $\widehat\lambda:\mathcal{P}([s])\rightarrow\mathcal{P}([s])$.

Notice that by definition one has $m_{\hat\lambda(\{i\})}=\lambda(m_i)$ for every $i\in[s]$. Therefore let $F=\{i_1,i_2,\ldots,i_j\}\subseteq[s]$, then
\begin{align*}
m_{\hat\lambda(F)}&=m_{\{\hat\lambda(\{i_1\}),\ldots,\hat\lambda(\{i_j\})\}}\\
&=\mathrm{Join}(m_{\hat\lambda(\{i_1\})},\ldots,m_{\hat\lambda(\{i_j\})})\\
&=\mathrm{Join}(\lambda(m_{i_1}),\ldots,\lambda(m_{i_j}))\\
&=\lambda(\mathrm{Join}(m_{i_1},\ldots,m_{i_j}))\\
&=\lambda(m_F).
\end{align*}

\end{remark}

In the following theorem, $\mathbf{x}$ and $\mathbf{x}'$ are sets of indeterminates.
\begin{theorem}\label{thm:IsoClasses}
Let $R=\Bbbk_\mathfrak{q}[\mathbf{x}]$ and $R'=\Bbbk_{\mathfrak{q}'}[\mathbf{x}']$ be skew polynomial rings with the same group of colors and same group bicharacter. Let $I$ and $I'$ be monomial ideals of $R$ and $R'$ respectively, minimally generated by $s$ monomials of degree at least 2. If there is a color preserving isomorphism of lattices $\lambda:L_I\rightarrow L_{I'}$ that induces an isomorphism of graphs $G_I\rightarrow G_{I'}$, then there is an isomorphism of graded color Lie algebras 
\[
\pi^{\geq2}(R/I)\cong\pi^{\geq2}(R'/I').
\]
\end{theorem}



\begin{proof}
Let $I=(m_1,\ldots,m_s)$ and $I'=(m_1',\ldots,m_s')$. Let $\mathbb{T}$ be the Taylor resolution of $R/I$ and $\mathbb{T}'$ the Taylor resolution of $R'/I'$. Let  $\{e_F\mid F\subseteq[s]\}$ be the canonical basis of $\mathbb{T}$, and $\{e_F'\mid F\subseteq[s]\}$ the canonical basis of $\mathbb{T}'$. Consider the map $\bar\lambda:\mathbb{T}\otimes_R\Bbbk\rightarrow \mathbb{T}'\otimes_{R'}\Bbbk$ defined by $e_F\otimes1\mapsto e_{\widehat\lambda(F)}'\otimes1$, we claim that $\bar\lambda$ is an isomorphism of color DG $\Gamma$-algebras over $\Bbbk$.

First we notice that up to relabeling the minimal generators of $I'$, we can assume that $\lambda(m_i)=m_i'$ for all $i\in[s]$. We notice that by the hypotheses on $\lambda$, it follows that $\bar\lambda$ preserves the internal degree and the color of the basis elements of the Taylor resolutions. To prove that $\bar\lambda$ commutes with the differentials, it suffices to show that
\[
\sum_{\substack{i\in F\\m_F=m_{F\backslash\{i\}}}}(e_{\widehat\lambda(F\backslash\{i\})}\otimes1)(-1)^{\sigma(F,i)}=\sum_{\substack{i\in \widehat\lambda(F)\\m_{\widehat\lambda(F)}=m_{\widehat\lambda(F)\backslash\{i\}}}}(e_{\widehat\lambda(F)\backslash\{i\}}'\otimes1)(-1)^{\sigma(\widehat\lambda(F),i)}.
\]
Since we relabeled the minimal generators of $I'$ in such a way that $\lambda(m_i)=m_i'$, it follows that $\widehat\lambda$ is the identity map, making the equality in the previous display obvious.

Next we show that $\bar\lambda$ preserves products. Let $F,G\subseteq[s]$ such that $F\cap G=\emptyset$ and $\mathrm{gcd}(m_F,m_G)=1$, so that $(e_F\otimes1)(e_G\otimes1)\neq0$ and $\bar\lambda(e_F\otimes1)\bar\lambda(e_G\otimes1)\neq0$. Then
\[
\bar\lambda((e_F\otimes1)(e_G\otimes1))=e_{\widehat\lambda(F\cup G)}'(-1)^{\sigma(F,G)}C(m_F,m_G),
\]
while
\[
\bar\lambda(e_F\otimes1)\bar\lambda(e_G\otimes1)=e_{\widehat\lambda(F)\cup\widehat\lambda(G)}'(-1)^{\sigma(\widehat\lambda(F),\widehat\lambda(G))}C(m_{\widehat\lambda(F)}',m_{\widehat\lambda(G)}').
\]
That the two previous displays are equal now follows from the fact that $\widehat\lambda$ is the identity map and from the assumption that $\lambda$ induces an isomorphism of weighted graphs $G_I\rightarrow G_{I'}$ (which implies that $C(m_F,m_G)=C(m_{\widehat\lambda(F)}',m_{\widehat\lambda(G)}')$).

Finally, the proof that $\bar\lambda$ preserves divided powers is analogous to the proof in the commutative case and it is therefore omitted. This concludes the proof that $\bar\lambda$ is an isomorphism of color DG $\Gamma$-algebras over $\kk$.

Next we show that $\pi^{\geq2}(R/I)\cong\pi^{\geq2}(R'/I')$. By Proposition \ref{prop:pi>=2}, it suffices to show that $\pi(\mathbb{T}\otimes_R\kk)\cong\pi(\mathbb{T}'\otimes_{R'}\kk)$, which follows from the fact that $\bar\lambda$ is an isomorphism of color DG $\Gamma$-algebras over $\kk$ and from Remark \ref{thm:piFunctorial}.
\end{proof}

\begin{remark}
We point out that the group of colors of the ring $\Bbbk_\mathfrak{q}[\mathbf{x}]$ is finite if and only if all the skew commuting parameters $q_{i,j}$ are roots of unity.
\end{remark}

\begin{corollary}\label{cor:FiniteIsoClasses}
Let $s$ be a natural number, $G$ a finite abelian group, and $\chi$ an alternating bicharacter on $G$. If $\mathbf{x}$ ranges over all finite sets of indeterminates, $\kk_\mathfrak{q}[\mathbf{x}]$ ranges among all skew polynomial rings with group of colors $G$ and bicharacter $\chi$, and $I$ ranges over all ideals in $\kk_\mathfrak{q}[\mathbf{x}]$ minimally generated by $s$ monomials of degree at least 2, then there exists only a finite number of isomorphism classes of the graded color Lie algebras $\pi^{\geq2}(\kk_\mathfrak{q}[\mathbf{x}]/I)$.
\end{corollary}

\begin{proof}
Let $I$ be a monomial ideal in $\kk[\mathbf{x}]$ generated by $s$ monomials. Then the number of vertices in the GCD graph $G_I$ is equal to the cardinality of the LCM lattice $L_I$, which is at most $2^s$. Since there are only finitely many possibilities for the edges of $G_I$ and finitely many possibilities for the labels (since the group of colors is finite), it follows that there are only finitely many possibilities for the pair $(L_I,G_I)$. The result now follows from the previous theorem.
\end{proof}

\begin{example}\label{example}
In this example we show that the hypothesis that $G$ is finite in the previous corollary is necessary. Indeed, let $n$ be a natural number $n\geq2$, and consider the rings $S_n=\Bbbk_q[x,y]/(x^n,y^n)$ where $xy=qyx$ and $q\in\Bbbk^*$ is not a root of unity. By \cite[Theorem 9.1]{FM}, it follows that
\[
\pi^{\geq2}(S_n)=\kk\theta_1\oplus\kk\theta_2,
\]
where the $G$-degree of $\theta_1$ (respectively $\theta_2$) is the inverse of the $G$-degree of $x^n$ (respectively the inverse of the $G$-degree of $y^n$), and both have cohomological degree 2. If $q$ is not a root of unity and if $n\neq m$, then $\pi^{\geq2}(S_n)$ and $\pi^{\geq2}(S_m)$ are not isomorphic as $G$-graded $\kk$-vector spaces and therefore they are not isomorphic as graded color Lie algebras.
\end{example}

\begin{remark}
Let $M$ be a finitely generated graded $S$-module with $S=\kk_\mathfrak{q}[\mathbf{x}]/I$, where $I$ is a monomial ideal. Let $\beta^S_i(M)$ be the rank of the free $S$-module in cohomological degree $i$ in the minimal free resolution of $M$. The \emph{Poincar\'{e} series} of $M$ is
\[
\mathrm{P}_M^S(t)=\sum_{i=0}^\infty\beta^S_i(M)t^i.
\]
It follows from \cite[Theorem 7.13]{FM} that
\[
\mathrm{P}_\kk^S(t)=\frac{\displaystyle\prod_{i=1}^\infty(1+t^{2i-1})^{\mathrm{rank}_\kk\pi^{2i-1}(S)}}{\displaystyle\prod_{i=1}^\infty(1-t^{2i})^{\mathrm{rank}_\kk\pi^{2i}(S)}}.
\]
We also point out that $\mathrm{rank}_\kk\pi^1(S)=|\mathbf{x}|$.
\end{remark}

The following corollary follows from \cref{{thm:IsoClasses}} and from the previous remark.
\begin{corollary}
Let $R=\Bbbk_\mathfrak{q}[\mathbf{x}]$ and $R'=\Bbbk_{\mathfrak{q}'}[\mathbf{x}']$ be skew polynomial rings with the same group of colors and same group bicharacter. Let $I$ and $I'$ be monomial ideals of $R$ and $R'$ respectively, minimally generated by $s$ monomials of degree at least 2. If there is a color preserving isomorphism of lattices $\lambda:L_I\rightarrow L_{I'}$ that induces an isomorphism of graphs $G_I\rightarrow G_{I'}$, then
\[
\frac{\mathrm{P}^{R/I}_\kk(t)}{(1+t)^{|\mathbf{x}|}}=\frac{\mathrm{P}^{R'/I'}_\kk(t)}{(1+t)^{|\mathbf{x}'|}}.
\]
\end{corollary}
Similarly we can also deduce the following
\begin{corollary}
Let $s$ be a natural number, $G$ a finite abelian group, and $\chi$ an alternating bicharacter on $G$. If $\mathbf{x}$ ranges over all finite sets of indeterminates, $\kk_\mathfrak{q}[\mathbf{x}]$ ranges among all skew polynomial rings with group of colors $G$ and bicharacter $\chi$, and $I$ ranges over all ideals in $\kk_\mathfrak{q}[\mathbf{x}]$ generated by $s$ monomials of degree at least 2, then there are only finitely many possibilities for the power series
\[
\frac{\mathrm{P}^{\kk_\mathfrak{q}[\mathbf{x}]/I}_\kk(t)}{(1+t)^{|\mathbf{x}|}}.
\]

\end{corollary}


\bibliographystyle{amsplain}

\bibliography{biblio}

\end{document}